\documentclass[12pt]{amsart}
\usepackage{graphicx,amsmath,amssymb,amsthm,paralist,anysize,todonotes}

\marginsize{2cm}{2cm}{1cm}{3cm}

\usepackage{tikz}
\usetikzlibrary{decorations.pathreplacing}
\usetikzlibrary{arrows}

\newcommand{\OA}{\mathbf{OA}}
\newcommand{\wt}{\mathbf{wt}}
\newcommand{\setF}{\mathbb{F}}
\newcommand{\setZ}{\mathbb{Z}}
\newcommand{\setR}{\mathbb{R}}

\newcommand\cP{{\mathcal P}}
\newcommand\cL{{\mathcal L}}
\newcommand\cV{{\mathcal V}}
\newcommand\cB{{\mathcal B}}

\newtheorem{theorem}{Theorem}[section]
\newtheorem*{theorem*}{Theorem}
\newtheorem*{maintheorem*}{Main Theorem}
\newtheorem{lemma}[theorem]{Lemma}
\newtheorem{proposition}[theorem]{Proposition}
\newtheorem{corollary}[theorem]{Corollary}
\newtheorem{conjecture}[theorem]{Conjecture}
\newtheorem*{conjecture*}{Conjecture}
\newtheorem{question}[theorem]{Question}

\newtheorem{asses}{Assumptions}

\begin{document}
\title[MMS Conjectures on Partial Geom.]{Manickam-Mikl{\'o}s-Singhi Conjectures on Partial Geometries} 

\author{Ferdinand Ihringer} 
\address{Einstein Institute of Mathematics,\\ 
Hebrew University of Jerusalem, Givat Ram, Jerusalem, 9190401, Israel}
\email{Ferdinand.Ihringer@uregina.ca}
\thanks{The first author acknowledges support from a PIMS Postdoctoral Fellowship.}
 
\author{Karen Meagher}
\address{Department of Mathematics and Statistics,\newline
University of Regina, 3737 Wascana Parkway, S4S 0A4 Regina SK, Canada}
\email{karen.meagher@uregina.ca} 
\thanks{Research supported in part by an NSERC Discovery Research Grant,
    Application No.:  RGPIN-341214-2013.}

\begin{abstract}
  In this paper we give a proof of the Manickam-Mikl{\'o}s-Singhi
  (MMS) conjecture for some partial geometries. Specifically, we give
  a condition on partial geometries which implies that the MMS conjecture
  holds. Further, several specific partial geometries that are
  counter-examples to the conjecture are described.
\end{abstract}


\maketitle

\section{Introduction}

In this paper we consider a version of the Manickam-Mikl{\'o}s-Singhi
conjecture (MMS conjecture) for partial geometries. This conjecture was
originally made in 1988, and, although recently substantial progress has been
made on the conjecture, the original statement is still open.

\begin{conjecture*}[{Manickam, Mikl{\'o}s, and Singhi \cite{Manickam1988,Manickam1988a}}]
  Let $f$ be a function from $\{ 1, \ldots, n \}$ to $\setR$ such that
  $\sum_{i=1}^n f(i)=0$.  Let $Y$ be the family of $k$-sets $y$ of $\{1, \ldots, n \}$ such
  that $\sum_{i \in y} f(i) \geq 0$.  If $n \geq
  4k$, then
  \begin{align*}
   |Y| \geq \binom{n-1}{k-1}.
  \end{align*}
  If equality holds, then $Y$ is the set of all $k$-subsets which contain a given element.
\end{conjecture*}

The MMS conjecture can be seen as a weighted version of the famous
Erd\H{o}s-Ko-Rado theorem (EKR theorem).

\begin{theorem*}[Erd\H{o}s, Ko, and Rado {\cite{ErdHos1961}}]
  If $n \geq 2k$, then a family $Y$ of pairwise intersecting $k$-subsets of $\{1, \ldots, n\}$ satisfies
  \begin{align*}
    |Y| \leq \binom{n-1}{k-1}.
  \end{align*}
  If equality holds and $n > 2k$, then $Y$ is the set of all $k$-subsets which contain a given element.
\end{theorem*}

A 2012 breakthrough paper by Alon, Huang, and Sudakov~\cite{Alon2012}
confirmed the conjecture for $n \geq 33k^2$.  More recently,
Chowdhury, Sarkis, and Shahriari \cite{Chowdhury2013} showed that the
MMS conjecture is true for $n \geq 8k^2$.  A linear bound was given by
Pokrovskiy~\cite{Pokrovskiy2013}, who proved that the MMS conjecture
holds, provided that $n \geq 10^{47}k$.

A family of $k$-sets from $\{ 1, \ldots, n \}$ is equivalent to a
$k$-uniform hypergraph.  Assume that there is a function $f: \{ 1, \ldots, n \}
\rightarrow \setR$ such that $\sum_{i=1}^n f(i)=0$, and say that an edge $y$ in
the hypergraph is nonnegative if $\sum_{i \in y} f(i) \geq 0 $. If the number of
nonnegative edges in a hypergraph is at least the minimal degree of a
vertex, then we say that the \textit{MMS star property} holds for the
hypergraph. Pokrovskiy's technique in~\cite{Pokrovskiy2013} relies on the
existence of hypergraphs with the MMS star property. A first step in
improving and simplifying his result is to show that the MMS star
property holds for more hypergraphs. The most general result known to
the authors on the MMS conjecture for hypergraphs is due to Huang and
Sudakov who proved the following.

\begin{theorem*}[Huang and Sudakov \cite{Huang2014}]\label{thm:huangsudakov}
  Let $H$ be a $k$-uniform hypergraph on $n$ vertices with constant codegrees.
  If $n > 10k^3$, then $H$ has the MMS star property.
\end{theorem*}

In 1988 Bier and Delsarte introduced the \textit{$i$-th distribution
  invariant} of an \textit{association scheme}
\cite{Bier1984,Bier1988}; this concept was essential in the
development of the MMS conjecture. In this paper we consider a version of the MMS conjecture for partial 
geometries. This corresponds to the first distribution invariant of the strongly 
regular graph associated to a partial geometry

The main result of Section~\ref{sec:generalbounds} yields the following
result.

\begin{theorem}\label{thm:mainbound}
  Let $(\cP,\cL)$ be a partial geometry of order $(s, t, \alpha)$. 
  Let $f$ be a function from $\cP$ to $\setR$ such that $\sum_{p \in \cP} f(p)=0$.
  Let $Y$ be the elements $y$ of $\cL$ such that $\sum_{p \in y} f(p) \geq 0$.
  Then there exists a function 
  $g(s, \alpha) \in o(\sqrt{s\alpha^2})$ such that if 
  \begin{align*}
   t > s\alpha + \sqrt{2s\alpha^2} + g(s, \alpha),
  \end{align*}
  then $|Y| \geq t+1$, i.e. $(\cP, \cL)$ has the MMS star property.
\end{theorem}

In general, the result by Huang-Sudakov~\cite{Huang2014} and our main
theorem cover different hypergraphs, but they correspond exactly for
partial geometries of order $(s, t, s+1)$. In this case, the result by
Huang-Sudakov yields the condition $t > 10s\alpha = 10s(s+1)$, which is slightly
worse than our condition. For $\alpha \neq s+1$ partial geometries,
seen as hypergraphs, have the codegrees $0$ and $1$, so
our objects are not handled in \cite{Huang2014}.

For all the mentioned MMS theorems it is interesting to determine the
tightness of the conditions.  We provide many examples of partial
geometries, for example some $2$-$(v, k, 1)$ designs, orthogonal arrays and
Latin squares, where the MMS star property does not hold. Furthermore,
one example in \cite{Huang2014}, which does not satisfy the MMS star
property, relies on the existence of Mersenne prime numbers, which
might mean that there are only finitely many of them. One of our
counterexamples covers that case and shows that counterexamples exist
for all prime powers.

\section{Preliminaries}\label{sec:prelim}

In this paper our goal is to consider versions of the MMS conjecture
for partial geometries.  A partial geometry is an incidence structure
$(\cP, \cL)$ where $\cL$ is a set of lines comprised of points from
the set $\cP$. The incidence structure $(\cP, \cL)$ must meet the
following conditions.

\begin{enumerate}[i.]
\item Each point is contained in the same number of lines. This is
  called the \textsl{replication number}, and it is denoted by $t+1$.
\item Every line contains the same number of points. This is the
\textsl{size} of the line, and it is denoted by $s+1$.
\item For any point $p$ not contained on a line $\ell$, the number of
  lines that contain $p$ and contain a point that is on $\ell$, does
  not depend on the choice of $\ell$ and $p$; we call this number
  $\alpha$.
\item Every pair of distinct points lies on at most one line.
\end{enumerate}
We refer to \cite[Section 1]{DeClerck1995} for an introduction to partial geometries.

Throughout this paper, $(\cP, \cL)$ denotes a partial geometry.
We say that $(\cP, \cL)$ is a partial geometry with parameters $(s,t,\alpha)$.
We say that two points are \textit{adjacent} if they lie on a common line,
and we say that two lines are \textit{adjacent} if they share a common point.

If $\alpha = 0$, then lines of the partial geometry are
disjoint. Such partial geometries are considered trivial and
trivially have the MMS star property. We will not consider
such partial geometries.

It is well known (and can be determined by a simple counting
argument~\cite[Section 5.6]{Godsil2015}) that
\[
| \cP | =\frac{(st+\alpha)(s+1)}{\alpha}, \qquad | \cL | =\frac{(st+\alpha)(t+1)}{\alpha}.
\]
For convenience, we will use $n$ for the number of points in $\cP$ and with this
notation
\[
\alpha = \frac{ts(s+1)}{n-s-1}.
\]

The set of all lines through a fixed point is called a
\textsl{star}; the size of a star is the replication number, $t+1$. We
say the partial geometry has the \textsl{EKR star property} if the
largest set of pairwise intersecting lines is a star. Moreover, the partial
geometry has the \textsl{strict EKR star property} if the stars are
the only sets of intersecting lines of maximum size. It is known that
any partial geometry has the EKR star property \cite[Section 5.6]{Godsil2015}.

For the MMS conjecture on a partial geometry, we assume that there is
a weighting that assigns a weight to each point in the partial
geometry. If the sum of all the weights on the points is zero (or
nonnegative), then the weighting is a zero-sum (or nonnegative)
weighting.  The weight of a line is defined to be the sum of the
weights of all the points on the line. The weight of a line $\ell$ is
denoted by $\wt(\ell)$.  We call a point or line with nonnegative (or
positive, or negative) weight a nonnegative (or positive, or negative)
point or line.

A set of lines in which each line has nonnegative weight is called an
\textsl{MMS set} for the weighting. The MMS conjecture predicts that
for any nonnegative weighting, the size of the largest MMS set is at
least the size of a star. Rather than considering the MMS conjecture
for all nonnegative weightings on the points, it is sufficient to only
consider the zero-sum weightings. We say a partial geometry has the
\textsl{MMS star property} if for any zero-sum weighting of the
points, the number of lines with positive weight is at least the
number of lines through a single point. Further, the partial geometry
has the \textsl{strict MMS star property} if for any zero-sum
weighting there exists a point for which all the lines through it have
nonnegative weight. We will show that any partial geometry, in which
$t$ is sufficient large relative to $s$ and $\alpha$, has the MMS star
property; that is, the MMS conjecture holds for these partial
geometries. First we will give an easy condition for when a partial
geometry has the MMS star property.


We define a
\textsl{spread} of a partial geometry $(\cP, \cL)$ to be a set of
disjoint lines such that each point belongs to exactly one of the lines in
the spread. A spread contains exactly $\frac{st+\alpha}{\alpha}$
lines. If it is possible to partition $\cL$
into disjoint spreads, then there will be exactly $t+1$
spreads in the partition. 
Any partial geometry that can be partitioned into disjoint spreads
trivially has both the MMS star property and the EKR star
property. This is well known, see~\cite[remarks following Theorem
4.5]{Manickam1988a}. 
This proof gives no indication if the partial geometry has either
the strict EKR star property or the strict MMS star property.

We want to determine conditions on a partial geometry to guarantee
that it also has the MMS star property.
An important tool for this is the following observation.

\begin{lemma}\label{lem:evaluebound}
For any line $\ell$ in a partial geometry with a zero-sum weighting
\begin{align}\label{eq:basicevalue}
(t-\alpha) \wt(\ell) = \sum_{s \sim \ell} \wt(s).
\end{align}
\end{lemma}
\begin{proof}
 Each point on $\ell$ lies on $t$ lines adjacent to $\ell$. Each point
 not on $\ell$ lies on $\alpha$ lines adjacent to $\ell$. From the definition
 of $\wt(\ell)$ it follows that
 \begin{align*}
    &\sum_{p \in \ell} \wt(p) = \wt(\ell) = -\sum_{p \notin \ell} \wt(p).
 \end{align*}
 This implies the assertion.
\qed \end{proof}

Lemma~\ref{lem:evaluebound} alone can be used to show that the
generalized quadrangles have the MMS star property.  The
\textsl{generalized quadrangles} are the partial geometries in which
$\alpha = 1$.

\begin{theorem}\label{thm:quads}
The generalized quadrangles have the MMS star property.
\end{theorem}
\begin{proof}
  Assume that $\ell$ is the line with the highest weight in a
  generalized quadrangle and that $\wt(\ell) = 1$. By
  Lemma~\ref{lem:evaluebound}, the sum of the weights of the lines
  adjacent to $\ell$ is $t-1$. We can assume that at least one of the
  lines adjacent to $\ell$ is negative (otherwise the result clearly
  holds).  So the sum of the weights of the nonnegative lines adjacent
  to $\ell$ must be strictly larger than $t-1$.  Since the weights of
  the lines are no more than $1$, there are at least $t$ lines
  adjacent to $\ell$ with nonnegative weight. These lines, together
  with $\ell$, form a set of $t+1$ lines with nonnegative weight.
\qed \end{proof}

\section{General bounds for partial geometries}
\label{sec:generalbounds}

In this section, we will prove our main theorem. 
We will assume that there is a zero-sum weighting on the points of
$(\cP, \cL)$. We can scale the weighting so that
the largest weight of a line is $1$. 
As Theorem~\ref{thm:quads} covers the case $\alpha =1 $, we
assume that $\alpha>1$.  
Let $\wt(\ell)$ denote the weight of a line.
Let $\ell_1$ denote the line with the highest weight and let $\ell_2$
denote the line with the second highest weight.
We will assume that the MMS
conjecture does not hold, so we assume that
there are at most $t$ lines (since the size of a star is $t+1$) with
nonnegative weight. We will show that this implies an upper bound
on $t$; hence if the bound is not met, then the partial geometry has the MMS
star property.

Note that the condition $ t > s\alpha + \sqrt{2s\alpha^2} + g(s, \alpha)$ in
Theorem~\ref{thm:mainbound} implies that $t > s\alpha
+ 3\alpha - 4$. Lemma~\ref{lem:evaluebound} implies that we have at least $2$
positive lines. Hence, we can make the following additional assumptions.
\begin{asses}\label{assumptions}
\begin{enumerate}[(i.)]
\item $P$ is the sum of the weights of the nonnegative lines;
\item $t > s\alpha + 3\alpha - 4$;
\item we have at least $2$ positive lines.
\end{enumerate}
\end{asses}

\begin{lemma}\label{lem:bnd_neg_weight_01}
Let $\ell$ be a negative line in $(\cP,\cL)$. Then the following holds.
\begin{enumerate}[(a)]
 \item 
 \begin{align*}
    (t-\alpha +1) \wt(\ell) - \sum_{ \stackrel{s \sim \ell}{\wt(s)\geq 0} } \wt(s) \geq -P.
  \end{align*}
 \item The line $\ell$ is not adjacent to all nonnegative lines.
\end{enumerate}
\end{lemma}
\begin{proof}
By Lemma~\ref{lem:evaluebound},
\begin{align*}
  (t-\alpha) \wt(\ell)   &= \sum_{s \sim \ell} \wt(s) \geq \sum_{\stackrel{s \neq \ell}{\wt(s) < 0}} \wt(s) = -P-\wt(\ell).
\end{align*}
This shows (a). Recall that we assume $t+1 > \alpha$, so $(t-\alpha +1) \wt(\ell) < 0$. This implies (b).
\qed \end{proof}

Together with our assumption $\wt(\ell_1) = 1$, Lemma \ref{lem:evaluebound} has two more helpful consequences.
\begin{lemma}\label{lem:bnd_pos_weight_01}
  We have $t-\alpha +1 <  P \leq t$.  \hfill{}\qed
 \end{lemma}

\begin{lemma}\label{lem:nhood_largest}
  If $\ell$ is a nonnegative line with $\wt(\ell') \leq \wt(\ell)$ for all lines $\ell'$ adjacent to $\ell$,
  then $\ell$ is adjacent to at least $t-\alpha + 1$ nonnegative
  lines. \hfill{}\qed
\end{lemma}

\begin{lemma}\label{lem:bnd_second_largest}
  The weight of $\ell_2$ is at least $\frac{t-\alpha}{t-1}$. 
\end{lemma}
\begin{proof}
  We have assumed that there are no more than $t$ lines with
  nonnegative weight and that the highest weight of a line is $1$.  By
  Lemma~\ref{lem:bnd_pos_weight_01}, the sum of the weights of all
  nonnegative lines is at least $t-\alpha+1$. So the average weight of
  all nonnegative lines, except the highest weight, is at least
  $\frac{t-\alpha}{t-1}$.
\qed \end{proof}

\begin{corollary}\label{cor:ell1_ell2_adj}
  The lines $\ell_1$ and $\ell_2$ meet in a point.
\end{corollary}
\begin{proof}
  Suppose that $\ell_1$ and $\ell_2$ are non-adjacent. 
  By Lemma \ref{lem:nhood_largest}, $\ell_1$ and $\ell_2$ are each
  adjacent to at least $t-\alpha+1$ nonnegative lines. Two non-adjacent
  lines in have exactly $(s+1)\alpha$ common
  neighbours. Thus the number of nonnegative lines that intersect at
  least one of $\ell_1$ or $\ell_2$ is at least $2(t-\alpha+1) -
  (s+1)\alpha$. But there are at most $t-2$ nonnegative lines, excluding
  $\ell_1$ and $\ell_2$, in the partial geometry, so
 \begin{align*}
    2(t-\alpha+1) - (s+1)\alpha \leq t-2.
  \end{align*}
  This is a contradiction.
\qed \end{proof}

In the following we denote the point $\ell_1 \cap \ell_2$ by $p$.

\begin{lemma}\label{lem:bnd_nhood_second_largest}
  The line with the second highest
  weight is adjacent to at least $t - \alpha$ nonnegative lines.
\end{lemma}
\begin{proof}
  Let $S$ be the set of nonnegative lines adjacent to $\ell_2$.  By
  Corollary~\ref{cor:ell1_ell2_adj}, $\ell_1$ is a
  neighbour of $\ell_2$. Then
   \begin{align*}
     (t-\alpha) \wt(\ell_2) \leq \sum_{s \in S} \wt(s)  \leq  1 + (|S|-1) \wt(\ell_2).
   \end{align*}
Rearranging, this becomes $(t-\alpha) - \wt(\ell_2)^{-1} + 1 \leq |S|$.
The bound from Lemma~\ref{lem:bnd_second_largest}, with the fact that
$t >  s\alpha + 3\alpha - 4 > 2\alpha-1$, implies that $\wt(\ell_2)^{-1} < 2$, which proves
the assertion.
\qed \end{proof}

\begin{lemma}\label{lem:unified_48_to_411}
  The point $p$ lies on at most $s\alpha + 2\alpha - s -2$ negative lines
  and there are at most $s\alpha - s + 2\alpha - 3$ nonnegative lines not on $p$.
\end{lemma}
\begin{proof}
  We will count the number $x$ of nonnegative lines
  adjacent to both $\ell_1$ and $\ell_2$.
  By Lemma~\ref{lem:nhood_largest}, $\ell_1$ has at least $t-\alpha+1$
  nonnegative neighbours.  Similarly, by
  Lemma~\ref{lem:bnd_nhood_second_largest}, $\ell_2$ has at least $t-
  \alpha$ nonnegative neighbours. In total, there
  are only $t$ nonnegative lines, so there cannot be more than $t$
  nonnegative lines that are adjacent to either $\ell_1$ or
  $\ell_2$. Using the principle of inclusion and exclusion, this
  implies that
\begin{align*}
(t-\alpha+1) + (t-\alpha) - x \leq t
\end{align*}
Rearranging, this becomes
\[
x  \geq t- 2\alpha+1.
\]
  
  There are exactly $s(\alpha-1) + (t-1)$ lines adjacent to both
  $\ell_1$ and $\ell_2$. Of these lines, exactly $t-1$ include
  $p$. So there are exactly $s(\alpha - 1)$
lines adjacent to both $\ell_1$ and $\ell_2$ that do not contain $p$.
There are at least $x$
nonnegative lines that are adjacent to both $\ell_1$ and $\ell_2$. Of these, at least
\[
x  - s(\alpha - 1) \geq 
 t + s - (s+2)\alpha + 1
\]
contain $p$. As there are $t+1$ lines on $p$, this implies the first part of the assertion.
At most $t$ are nonnegative, so at most
\begin{align*}
  t - (t+s-(s+2)\alpha+3)
\end{align*}
lines not on $p$ are nonnegative. This implies the second part of the assertion. \qed 
\end{proof}

In the following we denote the number of nonnegative lines not on $p$ by $m$.

\begin{lemma}\label{lem:BoundOnG_1}
  Any nonnegative line $g$, not on $p$, satisfies
\[
\wt(g) \leq \frac{\alpha}{t-\alpha-m+1}.
\]
\end{lemma}
\begin{proof}
We can assume that $g$ is the highest weight nonnegative line not on $p$.
Denote the set of nonnegative lines adjacent to $g$ by $S$,
then
\[
 (t-\alpha) \wt(g) = \sum_{s \sim g} \wt(s) 
 \leq  \sum_{s \in S} \wt(s) \
  =   \sum_{\stackrel{s \in S}{p \in s }} \wt(s) + \sum_{\stackrel{s \in S}{p \not \in s }} \wt(s). 
\]
In the partial geometry, there are exactly $\alpha$ lines that include
$p$ and meet the line $g$; each of these lines has maximum
weight $1$. Further, there are $m-1$ nonnegative lines that
meet $g$ and do not contain $p$; each of these has maximum
weight $\wt(g)$.  Thus, we have that
\[
 (t-\alpha) \wt(g)
\leq 
\alpha \max_{\stackrel{s \in S}{p \in s}} \wt(s) + (m-1)  \max_{\stackrel{s \in S}{p \not \in s}} \wt(s)
 \leq   
\alpha  + (m-1) \wt(g). 
\]
Rearranging shows the assertion.
\qed \end{proof}

\begin{lemma}\label{lem:BoundOnh_1}
  Any negative line $h$ on $p$ satisfies
\[
(t-\alpha + 1) \wt(h) \geq \frac{-m\alpha}{t-\alpha-m+1}.
\]
\end{lemma}
\begin{proof}
By Lemma~\ref{lem:BoundOnG_1},
\begin{align*}
\sum_{\stackrel{p \in \ell}{\wt(\ell)\geq 0}} \wt(\ell) 
            = P - \sum_{\stackrel{p \not \in \ell}{\wt(\ell)\geq 0}} \wt(\ell) 
           \geq P -  m \frac{\alpha}{t-\alpha-m+1}. 
\end{align*}
Since $\wt(h)$ is negative, Lemma~\ref{lem:bnd_neg_weight_01} implies
\begin{align*}
 (t-\alpha +1) \wt(h) \geq \sum_{\stackrel{s \sim h}{\wt(s) \geq 0}} \wt(s) -P 
                   \geq \left( P -  m \frac{\alpha}{t-\alpha-m+1} \right) - P. 
\end{align*}
Rearranging shows the assertion. \qed \end{proof}

\begin{lemma}\label{lem:too_negative_line}
  Let $g$ be the highest weight nonnegative line not on $p$.
  There exists a line with weight at most
  \begin{align*}
    -\frac{P-2m\,\wt(g)}{s(\alpha - 1)}.
  \end{align*}
\end{lemma}
\begin{proof}
By Lemma \ref{lem:bnd_neg_weight_01} (b), $g$ exists.
  In the notation of Lemma~\ref{lem:BoundOnh_1} let $h_1$ and $h_2$ be the two smallest weight
  lines on $p$, with $\wt(h_1) \leq \wt(h_2)$. By Lemma \ref{lem:bnd_neg_weight_01} (b), 
  this implies that both $h_1$ and $h_2$ are negative lines.
By Lemma~\ref{lem:unified_48_to_411}, the sum of the weights 
  of the nonnegative lines through $p$ is at least $P - m\wt(g)$.
  Let $S_i$ denote the set of negative lines adjacent to $h_i$.
  By Lemma~\ref{lem:evaluebound},
  \begin{align*}
(t-\alpha) \wt(h_i)  = \sum_{\stackrel{s \sim h_i}{\wt(s) \geq 0}} \wt(s)   +  \sum_{s \in S_i } \wt(s) 
                  \geq \left( P - m\wt(g) \right)   +  \sum_{s \in S_i} \wt(s).
  \end{align*}
  Rearranging, we have that 
\[
 \sum_{s \in S_i} \wt(s)  \leq  (t-\alpha) \wt(h_i)  - P + m\wt(g). 
\]
The sum of all negative lines is $-P$, so
\begin{align*}
  -P \leq (t-\alpha)(\wt(h_1)+\wt(h_2)) - 2P + 2m\wt(g) - \sum_{s \in S_1 \cap S_2} \wt(s).
\end{align*}

A line $s \in S_1 \cap S_2$ either contains $p$ or $s$ is one of the $s(\alpha-1)$ lines 
which meet $h_1$ and $h_2$, but does not contain $p$.
By Lemma~\ref{lem:unified_48_to_411}, the number of negative lines that
contain $p$ is at most $s\alpha + 2 \alpha - s - 2$ and these have at most weight $\wt(h_1)$.
Let $b$ denote the smallest weight line among the lines $S_1 \cap S_2$ which do not contain $p$.
As $t > s\alpha + 3\alpha - 4$ and $\wt(h_1), \wt(h_2) < 0$, it follows that
\begin{align*}
  -P \leq & (t-\alpha)(\wt(h_1)+\wt(h_2)) - 2P + 2m\wt(g) \\
        &-(s\alpha + 2 \alpha - s - 2)\wt(h_1) - s(\alpha-1)\wt(b)\\
  \leq &  - 2P + 2m\wt(g) - s(\alpha-1)\wt(b).
\end{align*}
Rearranging shows that the existence of $b$ implies the assertion.
\qed \end{proof}

The following result implies Theorem~\ref{thm:mainbound}.

\begin{proposition}\label{prop:nonmms_implies_small_t}
  If the partial geometry does not have the MMS star property, then
  \begin{align*}
t < \frac{1}{2} (\sqrt{(8\alpha^2-8\alpha)s+20\alpha^2-36\alpha+9}+(2\alpha-2)s+4\alpha-5).
  \end{align*}
\end{proposition}
\begin{proof}
  By Lemma \ref{lem:too_negative_line}, there exists a negative  line $b$ with 
  $\wt(b) s(\alpha-1) \leq 2m\,\wt(g) - P$. 
  By Lemma~\ref{lem:bnd_neg_weight_01}, we obtain.
  \begin{align*}
    (t-\alpha+1) \frac{2m\,\wt(g) - P}{s(\alpha - 1)} > -P.
  \end{align*}
  Solving for $P$, applying the lower bound in Lemma~\ref{lem:bnd_pos_weight_01} 
  and simplifying yields
  \begin{align}\label{eq:lowermwtg}
    t-\alpha - s\alpha + s + 1 < 2 m\,\wt(g).
  \end{align}
  Applying Lemma~\ref{lem:BoundOnG_1} gives an upper bound on
  $m\wt(g)$; since this upper bound is increasing in $m$,
  Lemma~\ref{lem:unified_48_to_411} can be applied to get the bound
  \begin{align}\label{eq:uppermwtg}
    m\wt(g) < \frac{\alpha m}{t - \alpha - m + 1} \leq \frac{\alpha (s\alpha - s + 2\alpha -3)}{t -s\alpha + s - 3\alpha + 4}.
  \end{align}
Putting equations~\eqref{eq:lowermwtg} and \eqref{eq:uppermwtg} together yields
\begin{align*}
(t-\alpha - s\alpha + s + 1)(t -s\alpha + s - 3\alpha + 4)  < 2 \alpha (s\alpha - s + 2\alpha -3).
\end{align*}
Solving for $t$ shows the assertion.
\qed \end{proof}

The same arguments show a similar result for the strict MMS star
property.

\begin{proposition}\label{prop:nonmms_implies_small_t_strict_version}
  Suppose that $t > s\alpha +3\alpha - 3$.
  If the partial geometry does not have the MMS strict star property, then
  \begin{align*}
    t < \frac{1}{2} (\sqrt{(8\alpha^2-8\alpha)s+20\alpha^2-28\alpha+9}+(2\alpha-2)s+4\alpha-5).
  \end{align*}
\end{proposition}

\section{Constructions of counterexamples for designs}

In this section we consider $2$-$(v,k,1)$ designs as a partial
geometries. The pair $(\cV, \cB)$ is a $2$-$(v,k,1)$ design if $\cV$ is
a base set of size $v$, usually assumed to be $\cV = \{1, \dots, v\}$,
and $\cB$ is a set of blocks. Each block is a subset of
$\cV$ with size $k$. The set of blocks has the property that every
pair of elements from $\cV$ occurs in exactly one block.

The number of blocks in a $2$-$(v,k,1)$ design is
\[
\frac{\binom{v}{2} }{\binom{k}{2}} = \frac{v(v-1)}{k(k-1)}
\]
and each element in $\cV$ belongs to exactly $\frac{v-1}{k-1}$ blocks.

Any $2$-$(v,k,1)$ design is a partial geometry; the points are the
elements of $\cV$ and the set of blocks $\cB$ are the lines. Each line
contains $k$ points, each point is in $\frac{v-1}{k-1}$ lines. If a
point $p$ is not on a block $B$, there are exactly $k$ blocks that
contain $p$ and intersect $B$. Thus a $2$-$(v,k,1)$ design is a
partial geometry with parameters $(k-1, \frac{v-1}{k-1}-1, k)$.

As we have $\alpha=k$, Proposition \ref{prop:nonmms_implies_small_t} implies that
all $2$-$(v, k, 1)$ designs with
\begin{align*}
  v > \frac{(k-1)\sqrt{8{k}^{3}+4{k}^{2}-28k+9}}{2} + (k-1)(k^2+9/2)
\end{align*}
have the MMS star property.

In this section, we give several examples of designs that do not have
the MMS star property. The first counterexample is a design that
contains a subdesign of the right size and a weighting with just two
weights. If $(\cV, \cB)$ is a $2$-$(v,k,1)$ design and $(\cV', \cB')$
is a $2$-$(w,k,1)$ design (where $w = |\cV'|$) with $\cV' \subset \cV$
and $\cB' \subset \cB$, then $(\cV', \cB')$ is a \textsl{subdesign} of
$(\cV, \cB)$.

\begin{theorem}\label{thm:designcounterwith2}
  Let $(\cV, \cB)$ be a $2$-$(v,k,1)$ design. If $(\cV, \cB)$ has a
  $2$-$(w,k,1)$ subdesign with $v < kw$ and $w^2-w < k(v-1)$,
then $(\cV, \cB)$ does not have the MMS star property.
\end{theorem}
\begin{proof}
  To prove this we build a weighting on $\cV$ that does not have
  $\frac{v-1}{k-1}$ nonnegative sets.  Assume that $(\cV', \cB')$ is
  the $2$-$(w,k,1)$ subdesign. Assign a weight of $|\cV|-|\cV'|$ to
  each element in $\cV'$, and a weight of $-|\cV'|$ to the remaining
  elements. Clearly this is a zero-sum weighting.

  The blocks contained in the subdesign will have weight
  $k(|\cV|-|\cV'|)$, which is positive. Any other block in $(\cV,
  \cB)$ will contain at most one element from $\cV'$, so will have weight at most
\[
|\cV|-|\cV'|  - (k-1) |\cV'| = |\cV|- k |\cV'|.
\]
If $v < kw$, then this is negative, and the only nonnegative blocks
under this weighting are the blocks of $(\cV', \cB')$.

There are only $\frac{w(w-1)}{k(k-1)}$ blocks in $(\cV, \cB)$, so if
\[
\frac{w(w-1)}{k(k-1)} < \frac{v-1}{k-1},
\]
then $(\cV', \cB')$ does not have the MMS star property.
\qed \end{proof}

For example, any $2$-$(19,3,1)$
design with a $2$-$(7,3,1)$ subdesign satisfies the condition of Theorem~\ref{thm:designcounterwith2}.  A
complete classification of designs with these parameters is given in
\cite{Kaski2008}. For $k=3$, these designs are the only ones that meet
the conditions in Theorem~\ref{thm:designcounterwith2}.

\begin{corollary}\label{cor:designstrictcounterwith2}
Let $(\cV, \cB)$ be a $2$-$(v,k,1)$ design. If $(\cV, \cB)$ has a $2$-$(w,k,1)$ subdesign with
$v < kw$ and $w^2-w \leq k(v-1)$,
then $(\cV, \cB)$ does not have the strict MMS star property.
\end{corollary}

The following theorem by Doyen and Wilson~\cite{Doyen1973} shows that
such designs exist.

\begin{theorem}\label{thm:DoyenWilson}
  A $2$-$(w,3,1)$ design can be embedded in a $2$-$(v,3,1)$ design if and only
  if $v \geq 2w+1$ and $v \equiv 1$ or $3 \pmod{6}$.
\end{theorem}

This means that the unique $2$-$(9,3,1)$ design can be embedded in a
$2$-$(25,3,1)$ design. From Corollary~\ref{cor:designstrictcounterwith2}, this
$2$-$(25,3,1)$ design does not have the strict MMS star property.

We can also construct examples with larger block size.  For $q$ a
prime power, there exists a $2$-$(q^3+q^2+q+1, q+1,1)$ design. The
points of the design are the $1$-dimensional subspaces of the
$4$-dimensional vector space $\setF_q^4$. The blocks are the
$2$-dimensional subspaces of $\setF_q^4$; specifically, each block
consists of all the $1$-dimensional subspaces of a $2$-dimensional
subspaces.  A $3$-dimensional subspace of $\setF_q^4$ induces a
$2$-$(q^2+q+1, q+1, 1)$ design.  This larger design does not have the
strict MMS star property by
Corollary~\ref{cor:designstrictcounterwith2}.

\begin{theorem}\label{thm:designcounterwith3}
  Let $(\cV, \cB)$ be a $2$-$(v,k,1)$ design. If $(\cV, \cB)$ contains
  a $2$-$(w,k,1)$ subdesign with $v< (k-1)(w-k)+k$ and
$v<\frac{(w-k)(w-1)}{k}+1$,
then $(\cV, \cB)$ does not have the MMS star property.
\end{theorem}
\begin{proof}
  Assume the elements of the $2$-$(w,k,1)$ subdesign are $\cV' =\{1,
  \dots, w\}$.  Assign the element $1$ a weight of 
\[
(v-w) - (w-1)\left(\frac{v-k}{w-k}\right)
\]
(note that this is a negative weight).
Assign a weight of $\frac{v-k}{w-k}$ to the remaining $w-1$ points
in the subdesign. All of the remaining points in the $2$-$(v,k,1)$
design are given a weight of $-1$.

Any block in the $2$-$(w,k,1)$ subdesign, that does not contain $1$, has
weight $\frac{k(v-k)}{w-k}$.  We claim that these are the only
nonnegative blocks. If a block in the $2$-$(w,k,1)$ subdesign contains $1$, then its weight is
\begin{align*}
(v-w) - (w-1)\left(\frac{v-k}{w-k}\right) + \frac{(k-1)(v-k)}{w-k} = -(w-k).\\
\end{align*}
Since $k<w$, this is negative. Any block not in the subdesign,
contains no more than one element from $\cV'$. Thus the maximum weight
of a block not contained in the subdesign is the weight of a block
that contains one element from $\cV'\backslash \{1\}$. The weight of
such a block is
\[
(k-1)(-1) + \frac{v-k}{w-k}.
\]
Since $v < (k-1)(w-k)+k$, this is negative.

Thus this is a weighting in which only the blocks from the subdesign
that do not contain $1$ are nonnegative. The number of such blocks is
\[
\frac{w(w-1)}{k(k-1)} - \frac{w-1}{k-1} = \frac{(w-k)(w-1)}{k(k-1)}.
\]
The size of a star in the design is $\frac{v-1}{k-1}$. Since $v <
\frac{(w-k)(w-1)}{k}+1$, the size of the set of nonnegative blocks is
strictly smaller than the size of a star.
\qed \end{proof}

There are examples of designs that meet the conditions in
Theorem~\ref{thm:designcounterwith3}. For example, from
Theorem~\ref{thm:DoyenWilson}, a $2$-$(9,3,1)$ design can be embedded
in either a $2$-$(19,3,1)$ design or $2$-$(21,3,1)$ design;
Theorem~\ref{thm:designcounterwith3} implies that neither of these
larger designs has the MMS star property.

\begin{lemma}\label{lem:counterexample_jungnickel_tonchev}
  For $q$ a prime power, there exists a $2$-$(q^3+q^2+q+1, q+1, 1)$
  design which does not have the MMS star property.
\end{lemma}
\begin{proof}
  Consider a $2$-$(q^3+q^2+q+1, q+1, 1)$ design $(\cV, \cB)$ that
  contains a set $S$ of $q^2+q$ blocks that is isomorphic to
  a $2$-$(q^2+q+1, q+1, 1)$ design missing one block $a$ such that $S$ cannot be
  extended to a $2$-$(q^2+q+1, q+1, 1)$ design contained in $(\cV, \cB)$.
  Such designs can be
  constructed by a careful use of the construction given in~\cite[Lemma
  1.1]{Jungnickel2010}.

Let $\delta$ denote the set of elements in the deleted block $a$.  Use
$\beta$ to denote the set of elements that are the blocks of $S$, but
are not in $a$. Let $\gamma$ be the remaining elements; these
are the elements in $\cV \setminus \cup_{b \in S}b$. Clearly,
\[
|\delta| =q+1, \quad |\beta| = q^2, \quad |\gamma| = q^3.
\]
Construct a zero-sum weighting of $(\cV, \cB)$ as follows, for $v \in \cV$
\begin{align*}
\wt(v) = 
\begin{cases}
q^2, & \textrm{if $v \in \delta$; } \\
(q+1)(q^3-1), & \textrm{if $ v \in \beta$;}\\
-q^2(q+1), & \textrm{if $ v \in \gamma$.}
\end{cases}
\end{align*}

Any block that is contained in $S$ will contain only elements of type $\delta$
or $\beta$, so the weight of these $q^2+q$ blocks will be nonnegative.
Any block not in $S$ will either contain one element from $\delta$ and
$q$ elements elements from $\gamma$, or at most $q$ elements from
$\delta$ and the remaining elements from $\gamma$. In the first case,
the weight of the block is
\[
(q+1)(q^3-1) + (- q) q^2(q+1) = -(q+1) ,
\]
and in the second case it is no more than
\[
-q^2(q+1) + q^3 = -q^2.
\]
Hence, exactly $q^2+q$ of the blocks of $(\cV, \cB)$ have
nonnegative weight, but the size of a star is
\[
\frac{q^3+q^2+q}{ q } = q^2+q+1.
\]
Thus the design $(\cV, \cB)$ does not have the MMS star property.
\qed \end{proof}


The final example in this section uses a generalization of
$2$-$(v,k,1)$ designs.  A pair $(\cV, \cB)$ is a $t$-$(v,k,\lambda)$
design if $\cV$ is a $v$ set (usually assumed to be $\{1, \dots, v\}$)
and $\cB$ is a set of blocks, each of size $k$, with the property that
every $t$-set of elements from $\cV$ occurs in exactly $\lambda$
blocks. The number of blocks in a $t$-$(v,k,\lambda)$ design and 
the replication number are
\begin{align*}
  \lambda\frac{\binom{v}{t}}{\binom{k}{t}} && \text{ and } && \lambda\frac{\binom{v-1}{t-1}}{\binom{k-1}{t-1}}.
\end{align*}

From any $t$-$(v,k,\lambda)$ design $(\cV, \cB)$, it is possible to
build another design called the \textsl{derived design}. This is
formed by fixing an element from $\cV$, then taking all the blocks
from $\cB$ that contain the fixed element. If the fixed element is
removed from these blocks, what remains is a
$(t-1)$-$(v-1,k-1,\lambda)$ design.

The Witt design is a well-known example of a $4$-$(23,7,1)$
design~\cite{Witt1937}, see also ~\cite[Section 11.4]{Brouwer1989}.
The parameter set of the derived design of the Witt design is
$3$-$(22,6,1)$. In~\cite[Section 5.4]{Godsil2015} it is shown that
this design has the strict EKR star property.

\begin{lemma}
The derived design of the Witt design does not have the MMS star property.
\end{lemma}
\begin{proof}
  Let $(\cV, \cB)$ be the derived design from the Witt design. This
  design has parameter set $3$-$(22,6,1)$, contains $77$ blocks, and
  the size of a star is $21$. This design has the property that the
  set of all blocks that are disjoint from a given block forms a
  $2$-$(16,6,2)$ design.

  To show that $(\cV, \cB)$ does not have the MMS star property, we
  will build a weighting of the elements such that fewer than $21$
  blocks have a nonnegative weight. Fix a block $B$ and consider the
  $2$-$(16,6,2)$ design that is formed by all the blocks that do not
  intersect $B$. Assign the elements in the
  $2$-$(16,6,2)$ design a weight of $6$ (these are the elements in the
  complement of $B$). Weight each of the elements in the fixed block
  $B$ by $-16$.  Clearly this forms a zero-sum weighting.

  The weight of the block $B$ is $6 (-16) = -96$.  A simple check shows
  that each block $(\cV, \cB)$, other than $B$, intersects $B$ in
  either $0$ or $2$ elements. The weight of any block that meets $B$
  in $2$ elements is $2(6) + 4\cdot(-16) = -52$. The weight of the
  blocks that do not intersect $B$ is $6(6) = 36$. Thus, these $16$
  blocks of the $2$-$(16,6,2)$ design are only nonnegative blocks with
  this weighting.
\qed \end{proof}

This example is interesting since the difference between the number of
nonnegative blocks and the size of a star is relatively large.

\section{Constructions of counterexamples for orthogonal arrays}

In this section we consider orthogonal arrays as partial
geometries. An orthogonal array $\OA(m,n)$ is an $m \times n^2$ array, and the entries
in the array are the elements from the set $\{0, \dots ,n-1\}$. This
set is called the \textsl{alphabet} of the orthogonal array. The array
has the property that for any two rows the pairs of entries in the
columns are never repeated; so between any two rows, each of the $n^2$
possible pairs of elements from the alphabet occurs in exactly one
column.

Orthogonal arrays are partial geometries. The lines of the partial
geometry are the columns of the array. The points are ordered pairs
$(r,i)$ with $r \in \{1,\dots, m\}$ and $i\in \{0,\dots, n-1\}$; if a
line contains the point $(r,i)$, then the column corresponding to the
line has the entry $i$ in row $r$. This partial geometry has $n^2$
lines and $mn$ points.  Each line contains exactly $m$ points, and
each point occurs in exactly $n$ lines.  For any point not on a line,
there are exactly $m-1$ lines that intersect the line and contain the
point. So an $\OA(m,n)$ is a partial geometry with parameters $(m-1, n-1, m-1)$;

Two columns in an orthogonal array intersect if they have the
same entry in the same row. A star is the set of all columns that have
an $i$ in position $r$; a star has size $n$.

By Proposition~\ref{prop:nonmms_implies_small_t}, if an $\OA(m,n)$ does
not have the MMS star property, then
\begin{align}\label{eq:boundforOAs}
n \leq 1/2 \left( \sqrt{8m^3-12m^2-36m+49} +2m^2-2m-5 \right).
\end{align}

\begin{lemma}\label{lem:OAcounterwith2}
 Let $A$ be an $\OA(m,n)$. If $A$ has a subarray that is an $\OA(m,n')$ with
$n <mn'$ and
$(n')^2<n$,
then $A$ does not have the MMS star property.
\end{lemma}
\begin{proof}
  Let $B$ be the subarray of $A$. There are $mn'$ points in $B$ and
  $m(n-n')$ points in $A$ that are not in $B$.

  Assign a weight of $m(n-n')$ to each of the points in $B$, and a
  weight of $-mn'$ to the remaining points in $A$. This is a zero-sum
  weighting. The weight of the columns in $B$ is positive. Any column
  not in $B$ contains at most one point from $B$. Thus its weight is no more than
\[
m(n-n') + (-mn')(m-1) = m ( n-mn') <0.
\]
So only the $(n')^2$ columns of $B$ are nonnegative. Since $(n')^2
<n$, this orthogonal array does not have the MMS star property.
\qed \end{proof}

In Section~\ref{sec:LatinSquares}, two examples of orthogonal arrays
that meet the conditions of Lemma~\ref{lem:OAcounterwith2} are given.

\begin{lemma}\label{lem:OAnotMMS}
  Let $A$ be an $\OA(m, (m-1)^2)$. Assume that $A$ contains a set of
  $(m-1)^2-1$ columns that is isomorphic to an $\OA(m, m-1)$ minus one
  column, and further assume that this column does not occur in
  $A$. Then $A$ does not have the MMS star property.
\end{lemma}
\begin{proof}
  Let $S$ be the set of columns that is isomorphic to an $\OA(m, m-1)$
  minus one column, call this column $C$. (We also assume that $C$ is
  not a column of $A$.)  Let $\delta$ be the set of the $m$ points in
  the column $C$, and let $\beta$ be the points in columns of $S
  \setminus C$.  Let $\gamma$ be the remaining points in $A$. Assign
  the following weights to the point $v$ of $A$
\begin{align*}
\wt(v) = 
\begin{cases}
 1,  &\textrm{if $v \in \delta$;} \\
  m(m-1) - \frac{1}{m-2}, & \textrm{if $v \in \beta$;}\\
   -m, & \textrm{if $v \in \gamma$.}\\
\end{cases}
\end{align*}
Since there are $m$ points of type $\delta$, $m(m-2)$ points of type $\beta$,
and $m(m-1)(m-2)$ points of type $\gamma$, this is a zero-sum
weighting.

Any column in $S$ contains only points of type $\delta$ or $\beta$, so
these $m^2-1$ columns all have nonnegative weight. Any column of $A$ that
is not in $S$ either contains one point from $\beta$ and $m-1$ points
of type $\gamma$, or at most $m-1$ points of type $\delta$ and the
remaining elements from $\gamma$. In the first case, the weight of the
column is
\[
 m(m-1) -\frac{1}{m-2}  + (m-1)(-m) = -\frac{1}{m-2},
\]
and in the second case, the weight is no more than
\[
 -m + (m-1)(1) = -1.  
\]
Hence, exactly the $(m-1)^2-1$ columns of $A$ have nonnegative
weight, while the size of a star is $(m-1)^2$.
\qed \end{proof}

Next we give a construction for orthogonal arrays that fulfills the
requirements of Lemma~\ref{lem:OAnotMMS}. To start, we state (without
proof) two well-known constructions of orthogonal arrays. This first
is the result that an $\OA(m+1,m)$ is equivalent to a projective plane
of order $m$~\cite[Section II.2.3]{Colbourn2007}.

\begin{theorem}\label{thm:primeposerOA}
For $m$ a prime power there exists an $\OA(m+1,m)$.
\end{theorem}

The next result is a simplified version of the \textsl{MacNeish
  construction} for orthogonal arrays~\cite{MacNeish1922}.  Let $A$ and $B$ be
two orthogonal arrays with $m$ rows and alphabet $n$. Denote the entry
in row $r$ and column $i$ of $A$ by $A_{r,i}$ (and similarly for
$B$). MacNeish's construction builds an $m \times n^2$ array as
follows.  Label the columns of the $m \times n^2$ array by the ordered
pairs $(i,j)$, where $i,j \in \{1,\dots,n\}$, and set the entry in row $r$ and column $(i,j)$ to be
$B_{r,i} + n A_{r,j}$.

\begin{theorem}[MacNeish's Construction]
  The $m \times n^2$ array constructed from MacNeish's construction is
  an orthogonal array, provided that $A$ and $B$ are both $m \times
  n^2$ orthogonal arrays.
\end{theorem}

Note that in this construction, we can assume without loss of
generality that the first column of $A$ is all zeros, then the
submatrix indexed by the columns $(i,1)$ is a copy of $B$.
See~\cite[Section 5.5]{Godsil2015} for an example
of an orthogonal array $\OA(3, 4)$ that does not have the strict EKR star property.

%

A variation of this construction
builds orthogonal arrays that meet the conditions of
Lemma~\ref{lem:OAnotMMS}. We call this construction the
\textsl{shifted MacNeish construction}.

\begin{lemma}\label{lem:OAexist}
  For any prime power $m$, there exists an $\OA(m, (m-1)^2)$ that contains a set
  of $(m-1)^2-1$ columns that is isomorphic to an $\OA(m, m-1)$ minus
  one column $C$, and that $C$ is not a column of $A$. 
\end{lemma}
\begin{proof}
  Let $A$ be an orthogonal array with $m$ rows and $(m-1)^2$ columns;
  this exists from Theorem~\ref{thm:primeposerOA} and $A$ is
  equivalent to the projective plane of order $m-1$. We may assume
  that the first column of $A$ is all zeros.  The first step in this
  construction is to change the first entry of the first row of $A$
  from $0$ to $m-1$, call this new array $B$.

  Next form an $\OA(m,(m-1)^2)$ by applying MacNeish's construction to
  $B$ and $A$. The columns of this $m \times (m-1)^4$ array will be
  indexed by ordered pairs $(i,j)$ with $i,j \in \{1,\dots, (m-1)^2\}$.
  The entry in row $r$ and column $(i,j)$ of this new array is
  $B_{r,i} + (m-1)A_{r,j}$. The entries in this array are the same as
  in MacNeish's construction, except that the entries in row $1$ and
  column $(1,i)$ for $i = 1,\dots, m-1$ will be $m-1$ larger (modulo
  $(m-1)^2$). So all the pairs from the alphabet still occur between the
  first row and any other row and this construction will produce an
  orthogonal array.

  The columns indexed by $(i,1)$ for $i = 1,\dots, m-1$ in the new
  array form an $\OA(m, m-1)$ minus one column $C$ (the all zeros
  column), and the all zeros column is not a column of the new
  array. Thus, by Lemma~\ref{lem:OAnotMMS}, the new $m \times (m-1)^4$
  orthogonal array does not have the MMS star property.
\qed \end{proof}

Below is an example of this construction. The arrays
\[
A = \left[
\begin{matrix}
0&0&1&1 \\
0&1&0&1 \\
0&1&1&0
\end{matrix}
\right],
\qquad 
B=\left[
\begin{matrix}
2&0&1&1 \\
0&1&0&1 \\
0&1&1&0
\end{matrix}
\right]
\]
can be used to construct the following $\OA(3,4)$, which does not have
the MMS star property
\begin{align}\label{eq:OA(3,4)counterexample}
\left[
\begin{matrix}
2011&2011&0233&0233 \\
0101&2323&0101&2323 \\
0110&2332&2332&0110 
\end{matrix}
\right].
\end{align}

\section{MMS star property for Latin squares}
\label{sec:LatinSquares}

The orthogonal arrays with three rows are a special subclass of
orthogonal arrays; any such orthogonal array corresponds to a Latin
square. If $(i,j,k)$ is any column in an orthogonal array with three
rows and an alphabet of size $n$, then it is possible to construct an
$n \times n$ Latin square by setting the $(i,j)$-entry of the square
to be $k$. An $\OA(3,n)$ is a partial geometry with parameters $(2,
n-1, 2)$.

The bound in Equation~\eqref{eq:boundforOAs} indicates that for $n\geq
8$ any $\OA(3,n)$ has the MMS star property. 
The array given at
(\ref{eq:OA(3,4)counterexample}) is an example of an $\OA(3,4)$ which
does not have the MMS star property, so it is clear that not every
$\OA(3,n)$ has the MMS star property. 
In this section we consider orthogonal arrays with three rows and an 
alphabet of size no larger than $7$.

A set of entries in the Latin square with one entry selected from each
row and each column is called a \textsl{transversal} if no two of the
entries are equal. A transversal is equivalent to a spread. If the
entries of a Latin square can be partitioned into disjoint
transversals, then the Latin square is said to be
\textsl{resolvable}. If a Latin square is resolvable, the disjoint
transversals of the Latin square form a partition of the points of the
orthogonal array into spreads. Hence, any resolvable Latin square has the MMS star property.

There is only one Latin square of order $2$; it (trivially) contains
an $\OA(3,1)$ subarray that meets the conditions of
Lemma~\ref{lem:OAcounterwith2}, so it does not have the MMS star
property. There is only one Latin square of order $3$, since it is
resolvable it has the MMS star property. There are two non-isomorphic
Latin squares of order $4$, one is resolvable, and hence has the MMS
star property. The other Latin square corresponds to the array given
in \ref{eq:OA(3,4)counterexample} and does not have the MMS star
property. There are also two Latin squares of order $5$, one is
resolvable and has the MMS star property. The other order $5$ Latin
square has a $2\times 2$ subsquare. By Lemma~\ref{lem:OAcounterwith2},
this Latin square does not have the MMS star property.

The Latin squares of order $6$ are more difficult than the smaller
arrays. There are $12$ nonisomorphic Latin squares of this size. Four of the
Latin squares of order $6$ have a subsquare that is a Latin square of
order $3$. Next we will see that this implies that they do not have
the strict MMS star property.

\begin{proposition}
  Let $A$ be an $\OA(3,6)$ and assume that $A$ contains an $\OA(3,3)$
  subarray, then $A$ does not have the strict MMS star property.
\end{proposition}
\begin{proof}
  Let $B$ be the $\OA(3,3)$ subarray in $A$.  There are $9$ elements
  of $A$ that do not occur in $B$; let $\gamma$ be the set of these
  elements. Pick any single point that occurs in $B$ and let $\beta$
  be the set that contains this one point. Let $\delta$ be the set of all
  points that occur in $B$, except the one point that is in $\beta$.

Define the following zero sum weighting of the points in $A$:
\begin{align*}
\wt(v) = 
\begin{cases}
\frac{7}{4}, & \textrm{if $v \in \delta$; } \\
-5, & \textrm{if $ v \in \beta$;}\\
-1, & \textrm{if $ v \in \gamma$.}
\end{cases}
\end{align*}

The $6$ columns of $B$ that do not contain the point in $\beta$ each
have weight $\frac{21}{4}$; we claim that these are the only
nonnegative columns in $A$.

Any other column of the orthogonal array will have at most one point
from $\delta$, so its weight will be no more than
$\frac{7}{4} + 2 (-1) <0$.
\qed \end{proof}

We have no examples of Latin squares of order $6$ with a weighting
with only $5$ nonnegative lines. Thus we make the following conjecture.

\begin{conjecture}
Any Latin square of order $6$ has the MMS star property.
\end{conjecture}

\begin{conjecture}
Any Latin square of order $7$ has the strict MMS star property.
\end{conjecture}

\section{Sporadic partial geometries with $\alpha=2$}

There are only three partial geometries with $\alpha=2$ known, all of them sporadic examples \cite{DeClerck2003}: Mathon's
partial geometry of order $(8, 20, 2)$ \cite{DeClerck2002}, van Lint and Schrijver's
partial geometry (\textit{vLS geometry}) of order $(5, 5, 2)$ \cite{Lint1981} and Haemers's partial geometry
of order $(4, 17, 2)$ \cite{Haemers1981}. For Mathon's partial geometry and Haemers's 
partial geometry Proposition \ref{prop:nonmms_implies_small_t_strict_version}
implies the strict MMS star property. We conclude this section by showing that 
the vLS geometry does not have the strict MMS star property.

We shall use the description for the vLS geometry by Cameron and van Lint \cite{Cameron1982}. Let
$G$ be the subgroup of $\setZ_3^6$ generated by $(1, 1, 1, 1, 1, 1)$.
We call the elements $a+G$ with $a \in \setZ_3^6$ and $\sum_{i=1}^6 a_i = 0$ \textit{points}
and the elements $a+G$ with $a \in \setZ_3^6$ and $\sum_{i=1}^6 a_i = 1$ \textit{lines}.
We say that $a+G$ and $b+G$ are incident if there exists a $c \in
\setZ_6^3$ with only one non-zero coordinate such that $a+G =
b+c+G$. The points and lines form a partial geometry of order $(5, 5,
2)$, that has $81$ points and $81$ lines.

Define a point $a+G$, or a line $b+G$, to be incident with $(1,1,0,0,0,0)+G$  
if there is a $c$, with only one non-zero entry, such that $a+G = c+(1,1,0,0,0,0)+G$, 
or $b= c+(1,1,0,0,0,0)+G$. Let $S$ be the set of lines of vLS incident with 
$(1,1,0,0,0,0)+G$. Note that if all the points incident to $(1,1,0,0,0,0)+G$ are 
removed from the lines in $S$, then what remains is a partial geometry $\Gamma$ of order 
$(4,  1,  2)$ that contains $15$ points and $6$ lines. Let $T$ denote the points
of $\Gamma$.  All lines from vLS not in $S$ meet $T$ in
at most one point.
Put the weight $66$ on the $15$ points of $T$ 
and put the weight $-15$ on the $66$ points not in $T$.  With this weighting
the lines in $S$ have weight $5 \cdot 66 - 15 > 0$, while the lines not in
$S$ have weight at most $66 - 5 \cdot 15 < 0$. So this is an MMS set with
$6$ lines that is not a star.

\section{Further work}

In all of our counterexamples to the MMS
conjecture, the largest MMS set is a set of intersecting lines. This
raises the question if this is always the case.

\begin{question}
  Is the smallest MMS set in a partial geometry always a set of pairwise intersecting lines?
\end{question}

Our counterexamples for the partial geometries frequently used a
substructure that formed a smaller partial geometry. These
substructures where chosen because they contain lines with many
intersections. Looking for sets of lines with a high level of
intersection is related to the EKR theorem. For partial geometries, it
is known that the EKR star property holds, but the characterization of the
largest intersecting lines is not known. This leads to our next question.

\begin{question}
  Is the maximum pairwise intersecting set of lines in a partial geometry
  always either a star or another partial geometry?
\end{question}
%
%

\paragraph*{\bf Acknowledgments} 
The authors would like to thank
Klaus Metsch for pointing out the construction by Jungnickel and Tonchev used in Lemma
\ref{lem:counterexample_jungnickel_tonchev}.  The authors would also
like to thank Ameera Chowdhury for discussing MMS conjectures with
them and providing various preprints of her work. The authors would like
to thank John Bamberg for his suggestion to include a discussion of all known
partial geometries with $\alpha=2$. The authors would like to thank the referees
for their very helpful and constructive comments on the presentation of the results.
%

\end{document}